\documentclass[10pt]{amsart}

\usepackage{amssymb,amsthm,amsmath}
\usepackage[numbers,sort&compress]{natbib}
\usepackage{color}
\usepackage{graphicx}
\usepackage{tikz}

\title[ ]{Proof of the HRT conjecture for almost every (1,3) configuration  }

\author{Wencai Liu}
\address[Wencai Liu]{Department of Mathematics, University of California, Irvine, California 92697-3875, USA} \email{liuwencai1226@gmail.com}


\newcommand{\R}{\mathbb{R}}
\newcommand{\Z}{\mathbb{Z}}

\newcommand{\C}{\mathbb{C}}
\theoremstyle{plain}
\newtheorem{theorem}{Theorem}[section]

\newtheorem{lemma}[theorem]{Lemma}

\theoremstyle{definition}

\newtheorem{remark}[theorem]{Remark}

\begin{document}


\begin{abstract}

We prove that    for almost  every  (1,3) configuration, there is no linear dependence between  the associated time-frequency translates of any $f\in L^2(\R)\backslash \{0\}$.
\end{abstract}
\maketitle

 \section{Introduction}

%

For a measurable function $f:\R\to \C$ and a subset $\Lambda\subset \R^2$, the associated Gabor system is
given by
\begin{equation*}
    \mathcal{G}(f,\Lambda)=\{M_bT_af:(a,b)\in \Lambda\},
\end{equation*}
where
\begin{equation*}
    M_bT_af(x)=e^{2\pi i bx}f(x-a).
\end{equation*}
We call $M_bT_af$ a time-frequency translate of $f$.

 The Heil-Ramanathan-Topiwala (HRT)  conjecture \cite{HRT96} asserts    that   finite  Gabor systems in $L^2(\R)$ are
linearly independent (also see \cite{heil2006}).  That is

{\bf The HRT Conjecture:} Let $\Lambda \subset \R^2$ be a finite set.     Then there is no non-trivial  function $  f \in L^2(\R)$ such that
the associated Gabor system $ \mathcal{G}(f,\Lambda)$ is linearly  dependent in $L^2(\R)$.

Here are  some examples to show that the $L^2(\R)$ property of function $f$  is essential.
\begin{description}
  \item[1] For any  trigonometric polynomial $f$, there exists a  finite  subset  $\Lambda\subset \R^2$ such that the associated Gabor system $ \mathcal{G}(f,\Lambda)$ is linearly dependent.
  \item[2] Let $f(x)=\frac{1}{2^n}$ for $x\in[n-1,n)$. Then $f\in L^2(\R^+)$ but $f\notin L^2(\R)$. It is easy to see that $\{f(x+1),f(x)\}$ is linearly dependent.
\end{description}

Since the formulation  of the HRT conjecture,   some results were  obtained  (see   \cite{HRTHS}  and references therein)
under   further restrictions  on the behavior of  function $f(x)$ at $x=\infty$ \cite{BB14,BS13,HRT96,BS15} or the structure of  the time-frequency  translates  $\Lambda$ \cite{HRT96,Lin99,DG12,BS2010,Dem10}.
%
Recall that  we call $\Lambda  $ an $(n,m)$ configuration if
 there exist 2 distinct parallel lines containing   $\Lambda$ such that one of them contains exactly $n$ points
 of $\Lambda$, and the other one contains exactly $m$  points of $\Lambda$.
  The following results hold  without  restriction on $f\in L^2(\R)\backslash\{0\}$.
\begin{itemize}
  \item
  $ \mathcal{G}(f,\Lambda)$  is linearly independent if  $\# \Lambda \leq 3$    or $\Lambda$ is colinear  \cite{HRT96}.
  \item
  $ \mathcal{G}(f,\Lambda)$  is linearly independent if  $\Lambda$ is a finite subset of a translate of a
lattice in $\R^2$ \cite{Lin99}. See \cite{DG12,BS2010} for  alternative proofs.
\item
$ \mathcal{G}(f,\Lambda)$  is linearly independent if  $\Lambda$  is a (2,2) configuration \cite{Dem10,DZ12}.
\item
$ \mathcal{G}(f,\Lambda)$  is linearly independent if   $\Lambda$  is a (1,3) configuration  with  certain  arithmetic restriction \cite{Dem10}. See Theorem \ref{Knowntheorem}.
\end{itemize}

In this paper,   we consider   $(1,3)$ configurations.
In \cite{Dem10}, Demeter proved
\begin{theorem}  \label{Knowntheorem}
The HRT conjecture holds for  special $(1,3)$ configurations
\begin{equation*}
   \Lambda=\{ (0,0),(\alpha,0),(\beta,0),(0,1)\},
\end{equation*}
\begin{itemize}
  \item [(a)] if there exists $\gamma>1$ such that
  \begin{equation}\label{DG1}
    \liminf_{n\to\infty}n^{\gamma}\min\{||n\frac{\beta}{\alpha}||,||n\frac{\alpha}{\beta}||\}<\infty,
  \end{equation}

  \item [(b)]if at least one of $\alpha,\beta$ is rational.
\end{itemize}

\end{theorem}

It is known that  $\{x\in\R: \text{ there exists some } \gamma>1 \text{ such that } \liminf_{n\to\infty}n^{\gamma} ||n x||<\infty\}$ is a set of zero  Lebesgue measure (e.g., \cite[Theorem 32]{continue}). Thus Theorem \ref{Knowntheorem} holds for a measure zero subset of parameters, and it has been an open problem to extend it to other
$(1,3)$ configurations.

Our main  result  is 

\begin{theorem}\label{mainth}
The HRT conjecture holds for  special $(1,3)$ configurations
\begin{equation*}
    \Lambda=\{  (0,0),(\alpha,0),(\beta,0),(0,1)\},
\end{equation*}
\begin{itemize}
  \item [(a)] if
  \begin{equation}\label{DG2}
    \liminf_{n\to\infty}n\ln n\min\{||n\frac{\beta}{\alpha}||,||n\frac{\alpha}{\beta}||\}<\infty,
  \end{equation}

  \item [(b)] if at least one of $\alpha,\beta$ is rational.
\end{itemize}


\end{theorem}
\begin{remark}
(b) of Theorem \ref{mainth} is the same  statement as (b) in Theorem \ref{Knowntheorem}. We list here for completeness.
\end{remark}
It is well known that $\{x\in\R:  \liminf_{n\to\infty}n\ln n ||n x||<\infty\}$ is a set of full  Lebesgue measure (e.g.,   \cite[Theorem 32]{continue}).   Then using  metaplectic transformations, we have the following theorem
\begin{theorem}\label{13almost}
Given any line $l$  in $\R^2$, let $(a_2,b_2)$  and $(a_3,b_3)$ be any two points   lying  in  $l$. Let $(a_1,b_1)$ be an  any point not lying in $l$.
Then  for almost every point  $(a_4,b_4)$ in $l$,
the HRT conjecture holds for the configuration $\Lambda=\{(a_1,b_1),(a_2,b_2),(a_3,b_3),(a_4,b_4)\}$.
\end{theorem}
\begin{proof}
By metaplectic transformations (see \cite{HRT96} for details), we can assume $l$ is  $x$-axis, and $(a_1,b_1)=(0,1)$, $(a_2,b_2)=(0,0)$ and $(a_3,b_3)=(\alpha,0)$. By Theorem \ref{mainth} and the fact that $\{x\in\R:  \liminf_{n\to\infty}n\ln n ||n x||<\infty\}$ is a set of full  Lebesgue measure,  we have
for almost every $\beta$, the HRT conjecture holds for $\Lambda=\{(0,1),(0,0),(\alpha,0),(\beta,0)\}$. We finish the proof.
\end{proof}
\section{The  framework  of the proof of Theorem \ref{mainth} }
 If $\frac{\alpha}{\beta}$ is a rational number,  it  reduces  to the  lattice case, which has been
proved \cite{Lin99}. Thus we also assume  $\frac{\alpha}{\beta}$ is irrational.

Assume Theorem \ref{mainth} does not hold. Then there exists some function $f$ satisfying
\begin{equation}\label{Gdecay}
   \lim_{|n|\to \infty} |f(x+n)|=0 \text{ a.e. } x\in[0,1)
\end{equation}
and nonzero $C_i\in \C$
such that
\begin{equation}\label{GRe}
   f(x+1)=f(x)(C_0+C_1e^{2\pi i \alpha x}+C_2e^{2\pi i\beta x})  \text{ a.e. } x\in \R.
\end{equation}
(Theorem \ref{mainth}  is  covered by the known results    if $C_i=0$ for some $i=0,1,2$)

Let
\begin{equation*}
    P(x)=C_0+C_1e^{2\pi i\alpha x}+C_2e^{2\pi i\beta x}.
\end{equation*}
For $n>0$, define
\begin{equation*}
    P_n(x)=\prod_{j=0}^{n} P(x+j),
\end{equation*}
\begin{equation*}
    P_{-n}(x)=\prod_{j=-n}^{-1} P(x+j).
\end{equation*}
Notice that $P(x+n)$  is an almost periodic function. Thus
for almost every $x\in[0,1)$,
\begin{equation*}
    P(x+n)\neq 0 \text{ for any } n\in \Z.
\end{equation*}

Iterating (\ref{GRe}) $n$ times on both sides (positive and  negative), we have
for $n>0$,
\begin{equation}\label{Gren}
    f(x+n)=P_n(x)f(x) \text{ a.e.  } x\in [0,1),
\end{equation}
and
\begin{equation}\label{Gre-n}
    f(x-n)=P_{-n}(x)^{-1}f(x) \text{ a.e.  } x\in [0,1).
\end{equation}
This implies that the value of function $f$ on $\R$ can be determined uniquely by its value on $[0,1)$ and function $P(x)$.

By Egoroff's   theorem and conditions (\ref{Gdecay}), \eqref{Gren} and \eqref{Gre-n}, there exists some positive Lebesgue  measure set $S\subset [0,1)$ and $d>0$,
such that
\begin{equation}\label{GF}
    \lim_{|n|\to \infty}f(x+n)=0 \text{ uniformly for } x\in S,
\end{equation}
\begin{equation}\label{GB}
   d <|f(x)| <d ^{-1} \text{  for all } x\in S,
\end{equation}
\begin{equation}\label{GT}
    f(x+n)=P_n(x)f(x) \text{ for  all } x\in S,
\end{equation}
and
\begin{equation}\label{GTI}
    f(x-n)=P_{-n}(x)^{-1}f(x) \text{ for all } x\in S.
\end{equation}
Demeter constructed a sequence $\{n_k\}\subset \Z^+$, such that
\begin{equation}\label{equ1}
   | P_{n_k}(x_k)P_{-n_k}(x^\prime_k)^{-1}|\geq C
\end{equation}
for some  $x_k,x^\prime_k\in S$. This   contradicts   (\ref{GF}), (\ref{GB}), (\ref{GT}) and (\ref{GTI}).

In order to complete the construction of  (\ref{equ1}), growth  condition (\ref{DG1}) was necessary in  \cite{Dem10}.
In the present paper, we follow the approach  of   \cite{Dem10}. The novelty of our work is in the subtler Diophantine  analysis.
This allows to make the restriction weak enough to obtain the result for a full Lebesgue measure set of parameters,
 and  significantly simplifies the  proof.

The rest of the paper is organized as follows.
In \S 3, we will give some basic  facts.
In \S 4, we give the proof of Theorem \ref{mainth}.
\section{Preliminaries}
We start with some basic notations.
 Denote by  $[x]$, $\{x\}$, $\|x\|$    the integer part, the  fractional part and the distance to the nearest integer of $x$. Let $\langle x\rangle$ be the unique number in $[-1/2,1/2)$ such that $x-\langle x\rangle$ is an integer.
 For  a measurable set $A\subset \R$,
  denote by $|A|$ its Lebesgue measure.

For  any irrational number $\alpha\in\R$, we define
$$ a_0=[\alpha],\alpha_0=\alpha,$$
and inductively for $k>0,$
\begin{equation}\label{A}
a_k=[ \alpha_{k-1}^{-1}], \alpha_k=\alpha_{k-1}^{-1}-a_k.
\end{equation}

We define
$$
\begin{array}{cc}
            p_0=a_0, & q_0=1, \\
            p_1=a_0a_1+1,& q_1=a_1 ,
          \end{array}
$$
and inductively,
\begin{eqnarray}
 \nonumber p_k &=& a_k p_{k-1}+p_{k-2}, \\
  q_k &=& a_k q_{k-1}+q_{k-2}.\label{B}
\end{eqnarray}
Recall that     $\{q_n\}_{n\in \mathbb{N}}$ is the sequence of   denominators of best approximations of irrational number $\alpha$,
since it satisfies
\begin{equation}\label{G29}
\forall 1\leq k <q_{n+1}, \| k\alpha\|\geq ||q_n\alpha||.
\end{equation}
Moreover, we also have the following estimate,

\begin{equation}\label{G210}
      \frac{1}{2q_{n+1}}\leq \|q_n\alpha\|\leq\frac{1}{q_{n+1}}.
\end{equation}
\begin{lemma}\label{Le3}
Let $ k_1< k_2< k_3<\cdots <k_m$ be a monotone integer sequence  such that $k_m-k_1< q_n$.
Suppose for some $x\in \R$
\begin{equation}\label{Gsmall}
    \min_{j=1,2\cdots,m}||k_j\alpha-x||\geq \frac{1}{4q_n}.
\end{equation}
Then
\begin{equation*}
    \sum _{j=1,2\cdots,m}\frac{1}{||k_j\alpha-x||}\leq Cq_n\ln q_n.
\end{equation*}
\end{lemma}
\begin{proof}
Recall that  $\langle x\rangle$ is the unique number in $[-1/2,1/2)$ such that $x-\langle x\rangle$ is an integer. Thus $||x||=|\langle x\rangle|$.
In order to prove the Lemma, it suffices to show that
\begin{equation}\label{Gsmallmay29}
    \sum _{j=1,2\cdots,m}\frac{1}{|\langle k_j\alpha-x\rangle|}\leq Cq_n\ln q_n.
\end{equation}
Let $S^+=\{j: j=1,2,\cdots,m , \langle k_j\alpha-x\rangle>0\}$.
Let $j_0^+$ be such that $j_0^+\in S^+$,  and
\begin{equation}\label{Gmay111}
   \langle k_{j_0^+}\alpha-x\rangle=\min_{j\in S^+}\langle k_j\alpha-x\rangle.
\end{equation}
By (\ref{G29}) and (\ref{G210}), one has for $ i\neq j $  and  $ i,j\in S^+$,
\begin{equation*}
 |\langle k_i\alpha-x\rangle-\langle k_j\alpha-x\rangle|  = ||(k_i\alpha-x)-(k_j\alpha-x)||\geq \frac{1}{2q_n}.
\end{equation*}
It implies the gap   between any two points $ \langle k_i\alpha-x \rangle$ and $\langle k_j\alpha-x \rangle$ with $ i,j\in S^+$ is larger than $\frac{1}{2q_n}$.
See the following figure.
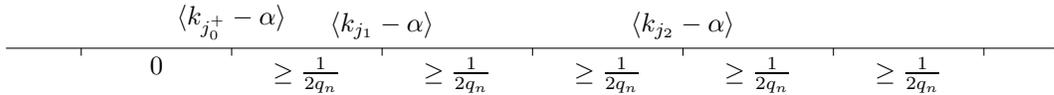
\begin{center}
\begin{tikzpicture}

\draw [->](-7,0)--(7,0);
\draw [-](0,0)--(0,-0.1);
\node [above] at (-4,0){ $ \langle k_{j_0^+}-\alpha \rangle$};
\node [below]at (1,0){ $\geq \frac{1}{2q_n}$};
\node [below]at (3,0){ $\geq \frac{1}{2q_n}$};
\node [below]at (5,0){ $\geq \frac{1}{2q_n}$};
\node [below]at (-1,0){ $\geq \frac{1}{2q_n}$};
\node [below]at (-3,0){ $\geq \frac{1}{2q_n}$};
\node [below]at (-5,0){ $0$};
\node [above] at (2,0){ $\langle k_{j_2}-\alpha \rangle$};
\node [above] at (-2,0){ $\langle k_{j_1}-\alpha \rangle$};

\draw [-](2,0)--(2,-0.1);
\draw [-](4,0)--(4,-0.1);
\draw [-](6,0)--(6,-0.1);
\draw [-](-2,0)--(-2,-0.1);
\draw [-](-4,0)--(-4,-0.1);
\draw [-](-6,0)--(-6,-0.1);

\end{tikzpicture}
\end{center}
It easy to see that the upper bound of  $ \sum _{j\in S^+}\frac{1}{||k_j\alpha-x||}$ is achieved if all the gaps are exactly  $ \frac{1}{2q_n}$.
In this case, the gap between the  $i$th closest points of $\langle k_j\alpha-x \rangle$ with $j\in S^+$   to $ \langle k_{j_0^+}\alpha-x \rangle$ is exactly $ \frac{i}{2q_n}$.
Thus by  (\ref{Gsmall}), we have
\begin{eqnarray}
 \sum _{j\in S^+}\frac{1}{||k_j\alpha-x||}&= &  \frac{1}{||k_{j_0^+}\alpha-x||} +\sum _{j\in S^+,j\neq j_0^+}\frac{1}{||k_{j}\alpha-x||}\nonumber \\
 &= &  \frac{1}{\langle k_{j_0}\alpha-x\rangle} +\sum _{j\in S^+,j\neq j_0^+}\frac{1}{\langle k_{j}\alpha-x\rangle}\nonumber \\
 &\leq & 4q_n+\sum_{1\leq j\leq q_n}  \frac{2q_n}{j}\nonumber \\
   &\leq& C q_n\ln q_n .\label{Gmay292}
\end{eqnarray}
Similarly, letting  $S^-=\{j: j=1,2,\cdots,m , \langle k_j\alpha-x\rangle<0\}$,
 one has
 \begin{equation}\label{Gmay291}
   \sum _{j\in S^-}\frac{1}{||k_j\alpha-x||}\leq C q_n\ln q_n.
 \end{equation}
 By \eqref{Gmay292} and \eqref{Gmay291}, we finish the proof.
\end{proof}
 Now we give two lemmas     which can be found in \cite{Dem10}.
 \begin{lemma}(\cite[Lemma 2.1]{Dem10})
\label{Le1}
Let $C_0,C_1,C_2\in \C\backslash \{0\}$. The polynomial $p(x,y)=C_0+C_1e^{2\pi i x}+C_2e^{2\pi i y}$ has at most two  real zeros $(\gamma_1^{(j)},\gamma_2^{(j)})\in [0,1)^2$, $j\in\{1,2\}$ and there  exists $t=t(C_0,C_1,C_2)\in \R\setminus \{0\}$ such that
\begin{equation}\label{jfaa1}
|p(x,y)|\geq C(C_0,C_1,C_2)\min_{j=1,2}(\| x-\gamma_1^j+t\langle y-\gamma_2^j\rangle\|+\| x-\gamma_1^j\|^2+\| y-\gamma_2^j\|^2),
\end{equation}
for each $x,y\in\R$.

\end{lemma}
\begin{remark}
In \eqref{jfaa1}, we assume $p(x,y)$  has two zeros.
If $p(x,y)$ has  one or no zeros,  we can proceed with our proof by  replacing  \eqref{jfaa1} with
\begin{equation*}
   |p(x,y)|\geq C(C_0,C_1,C_2)(\| x-\gamma_1+t\langle y-\gamma_2\rangle\|+\| x-\gamma_1\|^2+\| y-\gamma_2\|^2),
\end{equation*}
or
\begin{equation*}
   |p(x,y)|\geq C(C_0,C_1,C_2).
\end{equation*}

\end{remark}
\begin{lemma}(\cite[Lemma 4.1]{Dem10})\label{Le2}
Let $x_1,x_2,\ldots,x_N$ be $N$ not necessarily distinct real numbers.
Then for each $N\in \Z^+$ and each $\delta>0$, there exists a   set $E_{N,\delta}\subset [0,1)$  with $|E_{N,\delta}|\le \delta,$ such that
\begin{equation}
\label{equ2}
\sum_{n=1}^{N}\frac{1}{\|x-x_n\|}\leq C({\delta})N \log N,
\end{equation}
and
\begin{equation}
\label{equ3}
\sum_{n=1}^{N}\frac{1}{\|x-x_n\|^2} \leq C({\delta}) N^2,
\end{equation}
for each $x\in [0,1)\setminus E_{N,\delta}$.
\end{lemma}
\section{Proof of Theorems \ref{mainth}}

In this section,  $q_k,p_k,a_k$ are always  the coefficients of  the continued fraction expansion of  $\frac{\alpha}{\beta}$ as given in \eqref{A} and \eqref{B}.
Then condition
(\ref{DG2}) holds iff
\begin{equation}\label{DG3}
    \limsup_{k}\frac{a_k}{\ln q_k}>0,
\end{equation}
and also iff
\begin{equation*}
    \limsup_{k}\frac{q_{k+1}}{q_k\ln q_k}>0.
\end{equation*}
\begin{lemma}\label{Le4}
Suppose $\frac{\alpha}{\beta}$ is irrational and satisfies condition (\ref{DG2}). Then for
any  $s\in(0,1)$,
  there exists a sequence $N_k$ such that
  \begin{itemize}
  \item [(i)]
  \begin{equation}\label{DG4}
    N_k=m_{n_k}q_{n_k}, m_{n_k}\leq C(s),
  \end{equation}
  \item [(ii)]
  \begin{equation}\label{DG5}
    ||N_k\frac{\alpha}{\beta}|| \leq \frac{C(s)}{N_k\ln N_k},
  \end{equation}

and
  \item [(iii)]
  \begin{equation}\label{DG6}
   \{\frac{N_k}{\beta}\} \leq s.
  \end{equation}
\end{itemize}
\end{lemma}
\begin{proof}
By \eqref{DG3}, there exists a sequence  ${n_k}$ such that $ a_{n_k}\geq c \ln q_{n_k}$.
For any $s\in (0,1)$, let $m_{n_k}\in \Z^+$   be such that  $1\le m_{n_k}\le 1/s+1$ and  $N_k =m_{n_k}q_{n_k}$ satisfies (iii) (this can be done by the  pigeonhole principle). It is easy to check that $ N_k$ satisfies
condition (ii)  by the fact  $ a_{n_k}\geq c \ln q_{n_k}$.
\end{proof}
\begin{lemma}\label{Le5}
Let $C_0,C_1,C_2\in \C\backslash\{0\}$ and $\alpha,\beta$ be such that $\frac{\alpha}{\beta}$ is irrational. Let  $Q_k$ be a sequence such that $\gamma q_{n_k}\leq Q_k\leq \hat{\gamma} q_{n_k}$, where
$q_n$ is the continued fraction expansion of $\frac{\alpha}{\beta}$ and $\gamma,\hat{\gamma}$ are  constants.
Define
$$P(x)=C_0+C_1e^{2\pi i\alpha x}+C_2e^{2\pi i\beta x}.$$
Then for each  $\delta>0$, there exists a set $E_{Q_k,\delta}\subset [0,1)$ such that
$$|E_{Q_k,\delta}|<\delta$$
and
$$\sum_{n=0}^{Q_k-1}\frac{1}{|P(x+n)|}\leq C(\gamma,\hat{\gamma},\delta,C_0,C_1,C_2,\alpha,\beta) Q_k\ln Q_k$$
for each $x\in [0,1)\setminus E_{Q_k,\delta}$.
\end{lemma}
\begin{proof}
In order to make the proof simpler, we will use $C$ for constants  depending on  $\gamma,\hat{\gamma},\delta,C_0,C_1,C_2,\alpha,\beta$.

Let $(\gamma_1,\gamma_2)$ be a zero of the polynomial $p(x,y)=C_0+C_1e^{2\pi i x}+C_2e^{2\pi iy}$, and let $t$ be the real number given  by Lemma \ref{Le1}. Define
$$A_n(x):=\|\alpha (x+n)-\gamma_1+t\langle\beta (x+n)-\gamma_2\rangle\|+\|\alpha (x+n)-\gamma_1\|^2+\|\beta (x+n)-\gamma_2\|^2.$$
By Lemma \ref{Le1}, it suffices to find a set with $|E_{Q_k,\delta}|\le\delta,$ such that
\begin{equation}
\label{equ4}
\sum_{n=0}^{Q_k-1}\frac{1}{A_n(x)}\leq CQ_k\ln Q_k ,
\end{equation}
for each $x\in [0,1)\setminus E_{Q_k,\delta}$.

We distinguish between  two cases.

{\bf Case 1:}  $\alpha+t\beta\not=0$

In this case, one has
$$\|\alpha (x+n)-\gamma_1+t\langle \beta (x+n)-\gamma_2\rangle\|=\|(\alpha+t\beta)x+(\alpha+t\beta)n-\gamma_1-t\gamma_2-t[\beta (x+n)-\gamma_2]+mt\|,$$
where $m=-1$ if  $\{\beta (x+n)-\gamma_2\}>1/2$ and $m=0$ otherwise. We remind that $[\beta (x+n)-\gamma_2]$ is the integer part of $\beta (x+n)-\gamma_2$.

Note that the set
$$S:=\{(\alpha+t\beta)n-\gamma_1-t\gamma_2-t[\beta (x+n)-\gamma_2]+mt:\;x\in[0,1),\; 0\le n\le Q_k-1,\;m\in\{0,-1\}\}$$
has $O(Q_k)$ elements. By \eqref{equ2}  there exists some  $E^1_{Q_k,\delta}$ with $|E^1_{Q_k,\delta}|<\delta/2$  such that
$$ \sum_{y\in S}\frac{1}{\|(\alpha+t\beta)x+y\|}\leq C Q_k\ln Q_k$$
for each $x\in [0,1)\setminus E^1_{Q_k,\delta}$. This implies \eqref{equ4}.

{\bf Case 2:} $\alpha+t\beta=0$.

In this case, one has
$$\|\alpha (x+n)-\gamma_1+t\langle \beta (x+n)-\gamma_2\rangle\|=\|-\gamma_1-t\gamma_2+mt+\frac{\alpha}{\beta}[\beta (x+n)-\gamma_2]\|,$$
where $m$ is as before. Let $\xi$ be either $\gamma_1+t\gamma_2$ or $\gamma_1+t\gamma_2+t$, depending on whether $m=0$ or $-1$.  From Lemma \ref{Le3}, we have that for each $x\in [0,1)$
\begin{eqnarray}
  \sum_{n=0\atop{\|\frac{\alpha}{\beta}[\beta (x+n)-\gamma_2]-\xi\|\ge \frac{1}{4q_{n_k}}}}^{Q_k-1}\frac1{\|\frac{\alpha}{\beta}[\beta (x+n)-\gamma_2]-\xi\|}
   &\leq& C \sum_{n=0\atop{\|\frac{\alpha}{\beta}n-\xi\|\ge \frac{1}{4q_{n_k}}}}^{CQ_k}\frac1{\|\frac{\alpha}{\beta}n-\xi\|} \nonumber\\
  &\leq& C Q_k\ln Q_k  .\label{equ5}
\end{eqnarray}

Let $S(\xi)$ (not depending on $x$) be the set of those $0\le n\le Q_k-1$ such that $\|\frac{\alpha}{\beta}[\beta (x+n)-\gamma_2]-\xi\|\le \frac{1}{4q_{n_k}}$ for some $x\in[0,1)$. It is easy to see that  $\#  S(\xi)\leq    C$  by \eqref{G29} and \eqref{G210}.
 For $n\in S(\xi)$, we will use an alternative estimate
$$A_n(x)\ge \|\alpha (x+n)-\gamma_1\|^2.$$
By \eqref{equ3}, there exists some set $E^{2}_{Q_k,\delta}\subset [0,1)$  such that  $|E^{2}_{Q_k,\delta}|\leq \frac{\delta}{2}$
and
\begin{eqnarray}
  \sum_{n=0\atop{\|\frac{\alpha}{\beta}[\beta (x+n)-\gamma_2]-\xi\| \le \frac{1}{4q_{n_k}}}}^{Q_k-1}\frac{1}{A_n(x) }&\leq& C \sum_{n=0\atop{\|\frac{\alpha}{\beta}[\beta (x+n)-\gamma_2]-\xi\|\le \frac{1}{4q_{n_k}}}}^{Q_k-1}\frac{1}{\|\alpha (x+n)-\gamma_1\|^2} \nonumber\\
   &\leq& C(\delta),\label{equ6}
\end{eqnarray}
for each $x\in [0,1)\setminus E^{2}_{Q_k,\delta}$.
Thus in this case,  \eqref{equ4}   follows from \eqref{equ5} and \eqref{equ6}.
Putting two cases together, we finish the proof.

\end{proof}

\begin{theorem}
\label{keyth}
Under the conditions  of Lemma \ref{Le5},
let $N_k$ be a sequence such that (i), (ii) and (iii) in Lemma \ref{Le4} hold. Define $P_k:=\frac{N_k}{\beta}$ for $\beta>0$ and $P_k:=-\frac{N_k}{\beta}$ for $\beta<0$. Given $\delta>0$,  there exists $E_{k,\delta}\subset [0,1)$  with  $|E_{k,\delta}|\le \delta$ such that  for each $x,y$ satisfying $x\in [0,1)\setminus E_{k,\delta}$ and $x=y+P_k$, we have
$$|\prod_{n=0}^{[P_k]-1}P(y+n)|\le C(\delta, s,C_0,C_1,C_2,\alpha,\beta) |\prod_{n=0}^{[P_k]-1}P(x+n)|.$$
\end{theorem}
\begin{proof}
We write $C$ for $C(\delta,s,C_0,C_1,C_2,\alpha,\beta)$ again. Without loss of generality, we only consider the case $\beta>0$.

By \eqref{DG5}  we have
$$|e^{2\pi i \alpha x}-e^{2\pi i \alpha y}|=|e^{2\pi iN_k\frac{\alpha}{\beta}}-1|\le \frac{C}{P_k\ln P_k}$$
and
$$|e^{2\pi i\beta x}-e^{2\pi i\beta y}|=0.$$
Thus, for each $n\in \Z^+$, one has
$$|P(y+n)|\le |P(x+n)|+\frac{C}{P_k\ln P_k}.$$
By the fact $1+x\le  e^{x}$ for   $x>0$, we get
$$|P(y+n)|\le |P(x+n)|e^{\frac{C}{P_k\ln P_k|P(x+n)|}},$$
and thus
$$|\prod_{n=0}^{[P_k]-1}P(y+n)|\le |\prod_{n=0}^{[P_k]-1}P(x+n)|e^{\frac{C}{P_k\ln P_k}\sum_{n=0}^{[P_k]-1}\frac1{|P(x+n)|}}.$$
Now Theorem \ref{keyth}  follows from Lemma \ref{Le5}.
\end{proof}

\begin{proof}[\bf Proof of Theorem \ref{mainth}]
Suppose Theorem \ref{mainth}  is not true.
As argued in Section 2, there  there exist some function $f$,   a positive Lebesgue  measure set $S\subset [0,1)$ and $d>0$ such that
\eqref{GF}, \eqref{GB}, \eqref{GT} and \eqref{GTI} hold.
By the continuity of Lebesgue measure,  there exists $\varepsilon=\varepsilon(S)>0$ such that
\begin{equation*}
    |S\cup(\{P_k\}+S)|\leq \frac{101}{100}|S|,
\end{equation*}
for $\{P_k\}\leq \varepsilon$.
Let $\delta=\frac{|S|}{100}$.
Then $(S\setminus E_{k,\delta})\cap (\{P_k\}+ S)\not=\emptyset$.
Let   $s=\varepsilon$.
Applying Theorem \ref{keyth} with  $s$ and $\delta$,  we have
\begin{equation}
\label{e8}
|\prod_{n=0}^{[P_k]-1}P(y+n)|\le C |\prod_{n=0}^{[P_k]-1}P(x+n)|,
\end{equation}
for each  $x\in [0,1)\setminus E_{k,\delta}$ and $x=y+P_k$.

Now we can
 choose $x_k\in S\setminus E_{k,\delta}$ such that    $x_k-\{P_k\}\in S$.
Let $y_k=x_k'-[P_k]=x_k-P_k$. Then
\begin{equation}\label{jfaa2}
\prod_{n=0}^{[P_k]-1}P(y_k+n)=\prod_{n=1}^{[P_k]}P(x_k'-n).
\end{equation}

By \eqref{e8} and \eqref{jfaa2}, we get
\begin{equation*}
|\prod_{n=1}^{[P_k]}P(x_k'-n)|\le C |\prod_{n=0}^{[P_k]-1}P(x_k+n)|.
\end{equation*}
Applying  \eqref{GT} and \eqref{GTI} with $x_k,x^\prime_k\in S$, one has
\begin{equation}
\label{jfaa3}
f(x_k+[P_k])=f(x_k){\prod_{n=0}^{[P_k]-1}P(x_k+n)},
\end{equation}
and
\begin{equation}
\label{jfaa4}
f(x_k'-[P_k])=f(x_k')({\prod_{n=1}^{[P_k]}P(x_k'-n)})^{-1}.
\end{equation}
By  \eqref{GB},  \eqref{jfaa3} and \eqref{jfaa4},
 we obtain that
\begin{equation*}
|f(x_k+[P_k])f(x_k'-[P_k])|\geq \frac{d^2}{C}.
\end{equation*}
This  is  contradicted by \eqref{GF}, if we let $k\to\infty$.
\end{proof}

 \section*{Acknowledgments}
I would like to thank Svetlana
Jitomirskaya for   introducing to me the HRT conjecture   and inspiring discussions on this subject.
The author was supported by the AMS-Simons Travel
Grant 2016-2018 and NSF DMS-1700314. This research was also partially supported by
 NSF DMS-1401204.

\footnotesize
 \bibliographystyle{abbrv} 

\begin{thebibliography}{10}

\bibitem{BB14}
J.~J. Benedetto and A.~Bourouihiya.
\newblock Linear independence of finite {G}abor systems determined by behavior
  at infinity.
\newblock {\em J. Geom. Anal.}, 25(1):226--254, 2015.

\bibitem{BS13}
M.~Bownik and D.~Speegle.
\newblock Linear independence of time-frequency translates of functions with
  faster than exponential decay.
\newblock {\em Bull. Lond. Math. Soc.}, 45(3):554--566, 2013.

\bibitem{BS15}
M.~Bownik and D.~Speegle.
\newblock Linear independence of time-frequency translates in {$\Bbb{R}\sp d$}.
\newblock {\em J. Geom. Anal.}, 26(3):1678--1692, 2016.

\bibitem{BS2010}
M.~Bownik, D.~Speegle.
\newblock Linear independence of {P}arseval wavelets.
\newblock {\em Illinois Journal of Mathematics}, 54(2):771--785, 2010.

\bibitem{Dem10}
C.~Demeter.
\newblock Linear independence of time frequency translates for special
  configurations.
\newblock {\em Math. Res. Lett.}, 17(4):761--779, 2010.

\bibitem{DG12}
C.~Demeter and S.~Z. Gautam.
\newblock On the finite linear independence of lattice {G}abor systems.
\newblock {\em Proc. Amer. Math. Soc.}, 141(5):1735--1747, 2013.

\bibitem{DZ12}
C.~Demeter and A.~Zaharescu.
\newblock Proof of the {HRT} conjecture for {$(2,2)$} configurations.
\newblock {\em J. Math. Anal. Appl.}, 388(1):151--159, 2012.

\bibitem{heil2006}
C.~Heil.
\newblock Linear independence of finite Gabor systems.
\newblock In {\em Harmonic analysis and applications}, pages 171--206.
  Springer, 2006.

\bibitem{HRT96}
C.~Heil, J.~Ramanathan, and P.~Topiwala.
\newblock Linear independence of time-frequency translates.
\newblock {\em Proc. Amer. Math. Soc.}, 124(9):2787--2795, 1996.

\bibitem{HRTHS}
C.~Heil and D.~Speegle.
\newblock The {HRT} conjecture and the zero divisor conjecture for the
  {H}eisenberg group.
\newblock In {\em Excursions in harmonic analysis. {V}ol. 3}, Appl. Numer.
  Harmon. Anal., pages 159--176. Birkh\"auser/Springer, Cham, 2015.

\bibitem{continue}
A.~Y. Khinchin.
\newblock Continued fractions.
\newblock 1997.

\bibitem{Lin99}
P.~A. Linnell.
\newblock von {N}eumann algebras and linear independence of translates.
\newblock {\em Proc. Amer. Math. Soc.}, 127(11):3269--3277, 1999.

\end{thebibliography}

\end{document}